\documentclass{article}
\usepackage{amsfonts}
\usepackage{latexsym}
\usepackage{mathrsfs}
\usepackage{amssymb}
\usepackage{amsmath}
\usepackage{amsthm}
\usepackage{indentfirst}
\usepackage{color}
\usepackage{float}
\usepackage{graphicx}

\allowdisplaybreaks
\newtheorem{theorem}{\color{black}\indent Theorem}[section]
\newtheorem{lemma}{\color{black}\indent Lemma}[section]

\newtheorem{definition}{\color{black}\indent Definition}[section]

\begin{document}
\title{\LARGE\bf Existence of nontrivial solutions to a fourth-order Kirchhoff type elliptic equation with critical exponent}
\author{Qian Zhang\qquad Yuzhu Han$^{\dag}$}
 \date{}
 \maketitle

\footnotetext{\hspace{-1.9mm}$^\dag$Corresponding author.\\
Email addresses: zhangq21@mails.jlu.edu.cn (Q. Zhang), yzhan@jlu.edu.cn (Y. Han).

\thanks{
$^*$Supported by the National Key Research and Development Program of China
(grant no.2020YFA0714101).}}
\begin{center}
{\noindent\it\small School of Mathematics, Jilin University,
 Changchun 130012, P.R. China}
\end{center}

\date{}
\maketitle

{\bf Abstract}\ In this paper, a critical fourth-order Kirchhoff type elliptic
equation  with a subcritical perturbation is studied. The main feature of this
problem is that it involves both a nonlocal coefficient and a critical term,
which bring essential difficulty for the proof of the existence of weak solutions.
When the dimension of the space is smaller than or equals to $7$,
the existence of weak solution is obtained by combining the Mountain Pass Lemma
with some delicate estimate on the Talenti's functions. When the dimension of the
space is larger than or equals to $8$, the above argument no longer works.
By introducing an appropriate truncation on the nonlocal coefficient, it is shown that the problem
admits a nontrivial solution under appropriate conditions on the parameter.



\section{Introduction}
\setcounter{equation}{0}

In this paper, we consider the following critical fourth-order Kirchhoff type elliptic problem
\begin{equation}\label{eq1}
\ \ \ \ \ \ \ \ \begin{cases}
\Delta^2u-(a+b\int_{\Omega}|\nabla{u}|^2\mathrm{d}x) \Delta{u}=\lambda{f(x,u)}+|u|^{2^{**}-2}u,&x\in\Omega,\\
u=\dfrac{\partial u}{\partial\overrightarrow{n}}=0,&x\in\partial\Omega,
\end{cases}
\end{equation}
where $\Delta^2$ denotes the bi-harmonic operator, $\Omega\subset \mathbb{R}^N(N\geq5)$
is a bounded domain with smooth boundary $\partial\Omega$,
$\overrightarrow{n}$ is the unit outward normal on $\partial \Omega$, $2^{**}=\dfrac{2N}{N-4}$
is the critical Sobolev exponent, $a\geq0$, $b,\ \lambda>0$ and
$f:\overline{\Omega}\times\mathbb{R}\rightarrow \mathbb{R}$ is a continuous function that
satisfies some conditions which will be stated later.

Problem \eqref{eq1} is closely related to the stationary version of the following equation
$$u_{tt}+\Delta^2u-(a+b\int_{\Omega}|\nabla{u}|^2\mathrm{d}x)\Delta{u}=h(x,u), $$
which is regarded as a good approximation for describing nonlinear vibrations of beams or plates.
Moreover, it is also widely applied in engineering, physics and other applied sciences (see\cite{Ball,Berger}).

It is well known that the first equation in \eqref{eq1} is usually referred to as
being nonlocal due to the presence of the term $a+b\int_{\Omega}|\nabla{u}|^2\mathrm{d}x$,
which implies that the equation is no longer a pointwise identity.
On the other hand, the equation involves a critical term in the sense that
the embedding $H_0^2(\Omega)\hookrightarrow L^{2^{**}}(\Omega)$ is not compact.
These features cause some mathematical challenges which makes it difficult to
prove the existence of weak solutions. It is also because of these challenges
that such problems are attracting more and more people, especially after the
pioneering work of Br\'{e}zis and Nirenberg \cite{BreNi}.
We only briefly recall some related works on nonlocal problems.

In \cite{Figueiredo2013}, Figueiredo studied the following Kirchhoff type problem
with critical exponent
\begin{equation}\label{eq1.3}
\ \ \ \ \ \ \ \ \begin{cases}
-M(\int_{\Omega}|\nabla u|^2\mathrm{d}x)\Delta u=\lambda g(x,u)+|u|^{2^{*}-2}u, \ &x\in\Omega,\\
u(x)=0, \ &x\in\partial\Omega,
\end{cases}
\end{equation}
where $\Omega\subset \mathbb{R}^N(N\geq 3)$ is a bounded smooth domain, $\lambda$ is a positive
parameter, $M$ and $f$ are continuous functions that satisfy some specific conditions.
On the basis of the variational method, truncation argument and a priori estimate,
he obtained the existence of a positive solution to problem \eqref{eq1.3} and
studied the asymptotic behavior of this solution as $\lambda\rightarrow+\infty$.
Later, Naimen \cite{ND} investigated the existence of positive solutions to problem \eqref{eq1.3}
with $N=3$ and $M(t)=a+bt$. The existence of a positive solution is obtained for
all $\lambda>0$ by means of variational method, when $g(x,s)$ satisfies some
growth conditions as $s\rightarrow+\infty$. When $g(x,s)$ fulfills another kind of growth condition,
the existence of a positive solution is proved for suitably large $\lambda$,
with the help of the variational method and truncation argument.

There are also some works on fourth-order Kirchhoff type problems. For example,
Wang et al. \cite{Wang and An} considered the following boundary value problem
\begin{equation}\label{eq1.2}
\ \ \ \ \ \ \ \ \begin{cases}
\Delta^2u-M(\int_{\Omega}|\nabla u|^2\mathrm{d}x)\Delta u=f(x,u), \ &x\in\Omega,\\
u=\Delta u=0,&x\in\partial\Omega,
\end{cases}
\end{equation}
where $f(x,s)$ is a subcritical nonlinearity. With the help of the Mountain Pass Lemma,
they investigated the existence and multiplicity
of nontrivial solutions to problem \eqref{eq1.2} under appropriate conditions on the function
$M(t)$ and the nonlinearity $f(x,u)$. Later, problem \eqref{eq1.2} with $M(t)=\lambda(a+bt)$
was reconsidered by them and at least one nontrivial solution was shown to exist when $\lambda$
is bounded from above (see \cite{WangFanglei}).

As for the critical fourth-order Kirchhoff type problems,
Hssini et al. \cite{HE2016} studied the following problem
\begin{equation}\label{eq3}
\ \ \ \ \ \ \ \ \begin{cases}
\Delta^2u-M(\int_{\Omega}|\nabla u|^2\mathrm{d}x)\Delta u=\lambda f(x,u)+|u|^{2^*-2}u, \ &x\in\Omega,\\
u=\Delta u=0,&x\in\partial\Omega,
\end{cases}
\end{equation}
where $\Omega\subset \mathbb{R}^N(N\geq 5)$ is a bounded smooth domain
and $\lambda$ is a positive parameter. They used variational method to
show that under some conditions on $M$ and $f$, there exists
$\lambda^*>0$ such that problem \eqref{eq3} has a nontrivial solution
for all $\lambda\geq\lambda^*$. Later, Song et al. \cite{Shishaoyun}
considered the multiplicity of solutions to problem \eqref{eq3} by
using the concentration compactness principle and variational method,
for different kinds of subcritical perturbations.

By carefully checking the conditions in \cite{HE2016,Shishaoyun}, one sees that for
the special case $M(t)=a+bt$ with $a\geq0$ and $b>0$, the admissible dimension of the
space must satisfy $5\leq N\leq7$, which means that there is almost no conclusion for
problem \eqref{eq1} when $N\geq8$. Motivated mainly by \cite{HE2016,Shishaoyun} and
to make our results complete, we intend to study the existence of nontrivial solutions
to problem \eqref{eq1} for all $N\geq5$. When $5\leq N\leq7$, we can prove that problem \eqref{eq1}
admits at least one nontrivial solution for all $\lambda>0$, by using the Mountain Pass
Lemma and the concentration compactness principle. A crucial step in this process is to
show for each $\lambda>0$ that the mountain pass level is smaller than a critical value $c$,
which is done after some delicate estimate on the Talenti's function. However, the above argument
does not work when $N\geq8$. The main reason is that for this case the boundedness of the (PS)
sequence can not be obtained in general since $2<2^{**}\leq4$. In view of this, we introduce an
appropriate truncation on the nonlocal coefficient to overcome this difficulty, and
show that there exists a constant $\lambda_*>0$ such that problem \eqref{eq1} admits at least one
nontrivial solution for all $\lambda\geq\lambda_*$.  Moreover, we investigate the asymptotic behavior
of the solution of problem \eqref{eq1} when $\lambda$ converges to infinity by using a priori estimate.

Before stating the main results, let us introduce some structural conditions
imposed on the nonlinearity $f(x,s)$.

$(f1)$ It holds that $0\leq\theta F(x,s)\leq f(x,s)s$ for all $x\in\Omega$ and $s\in \mathbb{R}$,
where $4\leq\theta<2^{**}$ when $5\leq N\leq7$, $2<\theta<2^{**}\leq4$ when
$N\geq8$, and $F(x,s):=\int_0^s f(x,t)\mathrm{d}t$.

$(f2)$ $f(x,s)=o(|s|)$ as $s\rightarrow0$ and $f(x,s)=o(|s|^{2^{**}-1})$ as $s\rightarrow \infty$
uniformly for all $x\in\Omega$.

$(f3)$ For $5\leq N\leq 7$, there exists a nonempty open set $\omega\subset\Omega$ such that
$$\lim\limits_{s\rightarrow +\infty}\frac{f(x,s)}{s^{N^2-14N+52}}=+\infty$$uniformly for $x\in\omega$.

$(f4)$ For $N\geq 8$, there exist a nonempty open set $\widetilde{\omega}\subset\Omega$ and an interval $I\subset(0,\infty)$
such that $f(x,s)>0$ for $x\in\widetilde{\omega}$ and $s\in I$.\\

Our results can be summarized as follows:
\begin{theorem}\label{th1.1}
For $5\leq N\leq 7$, let $a\geq0$, $b>0$ and $\lambda>0$. Assume that $f$ satisfies $(f1)-(f3)$.
Then problem \eqref{eq1} has at least one nontrivial solution for all $\lambda>0$.
\end{theorem}

\begin{theorem}\label{th1.2}
For $N\geq8$, let $a>0$, $b>0$ and $\lambda>0$. Assume that $f$ satisfies $(f1)$, $(f2)$ and $(f4)$.
Then there exists a constant $\lambda_*>0$ such that problem \eqref{eq1} has at least
one nontrivial solution for all $\lambda\geq\lambda_*$. Moreover, if $u_\lambda$ is
a nontrivial solution to problem \eqref{eq1}, then $\lim\limits_{\lambda\rightarrow\infty} \|u_\lambda\|=0$.
\end{theorem}

This paper is organized as follows. In Section 2, we introduce some notations, definitions and necessary lemmas.
The existence of a nontrivial solution to problem \eqref{eq1} for $5\leq N\leq 7$ is proved in Section $3$.
In Section 4, we prove the existence and the asymptotic behavior of a nontrivial solution to problem \eqref{eq1} with $N\geq8$.

\par
\section{Preliminaries}
\setcounter{equation}{0}

In this section, we begin with some notations and definitions that will be used throughout the paper.
We denote by $\|\cdot\|_p$ the usual $L^p(\Omega)$-norm for $1\leq p\leq\infty$ and use $\rightarrow$
and $\rightharpoonup$ to denote the strong and weak convergence in each Banach space, respectively.
In addition, we use $C,C_1,C_2,\cdots$ to denote (possibly different) generic positive constants,
and use $C(\delta)$ to mean a positive constant depending on $\delta$. We always denote by $B(x_0,R)$
the ball with radius $R$ centered at $x_0$ in $\mathbb{R}^N$, and use $\omega_N$ to denote the area
of the unit sphere in $\mathbb{R}^N$. Sometimes $B(0,R)$ will be simply written as $B_R$ if no
confusion arises. For each $t>0$, $O(t)$ denotes the quantity satisfying
$|\frac{O(t)}{t}|\leq C$, and $o(t)$ means that $|\frac{o(t)}{t}|\rightarrow 0$ as $t\rightarrow 0$.
Finally, the Hilbert space $H_0^2(\Omega)$ will be equipped with
the inner product
$$(u,v)=\int_{\Omega}(\Delta u\Delta v +a\nabla u\nabla v)\mathrm{d}x,$$
and the norm
$$\|u\|^2=\int_{\Omega}(|\Delta{u}|^2+a|\nabla u|^2)\mathrm{d}x,$$
and its dual space is denoted by $H^{-2}(\Omega)$.
Since $a\geq 0$, it is easy to see that the norm $\|\cdot\|$ is equivalent to the full one due to Poincar\'{e}'s inequality.

In this paper, we consider weak solutions to problem \eqref{eq1} in the following sense.
\begin{definition}\label{weaksolution}$\mathrm{\bf{(Weak \ solution)}}$
A function $u\in H_0^2(\Omega)$ is called a weak solution to problem \eqref{eq1}, if for all $v\in H_0^2(\Omega)$, it holds that
$$\int_{\Omega}\Delta u\Delta v\mathrm{d}x+(a+b\int_{\Omega}|\nabla u|^2\mathrm{d}x)\int_{\Omega}\nabla u\nabla v\mathrm{d}x-\lambda\int_{\Omega}f(x,u)v\mathrm{d}x-\int_{\Omega}|u|^{2^{**}-2}uv\mathrm{d}x=0.$$
\end{definition}

The energy functional associated with problem \eqref{eq1} is
$$J(u)=\frac{1}{2}\|u\|^2+\frac{b}{4}(\int_{\Omega}|\nabla u|^2\mathrm{d}x)^2-\lambda\int_{\Omega}F(x,u)\mathrm{d}x-\frac{1}{2^{**}}\int_{\Omega}|u|^{2^{**}}\mathrm{d}x, ~~~~~~\forall\ u\in H_0^2(\Omega).$$
Obviously, $J(u)$ is a $C^1$ functional in $H_0^2(\Omega)$(see \cite{Willem}).

We use $J':H_0^2(\Omega)\rightarrow H^{-2}(\Omega)$ to denote the Frech\'{e}t derivative operator of $J$,
which satiafies
\begin{equation*}
\begin{split}
\langle J'(u),v\rangle=&\int_{\Omega}\Delta u\Delta v\mathrm{d}x+(a+b\int_{\Omega}|\nabla u|^2\mathrm{d}x)\int_{\Omega}\nabla u\nabla v\mathrm{d}x-\lambda\int_{\Omega}f(x,u)v\mathrm{d}x\\
&-\int_{\Omega}|u|^{2^{**}-2}uv\mathrm{d}x, \ \ \ \ \ \ \ \ \ \ \ \ \ \ \ \forall\ u,v\in H_0^2(\Omega).
\end{split}
\end{equation*}

It is well known that every critical point of $J$ is a weak solution to problem \eqref{eq1}.
Hence in order to find a nontrivial solution to \eqref{eq1} we need only to find a nontrivial
critical point of $J$. This is done by combining the following compactness condition,
which is usually referred to as the $(PS)_c$ condition or the local $(PS)$ condition, with
the Mountain Pass Lemma.

\begin{definition}\label{PS}($(PS)_c$ condition)
Assume that $X$ is a real Banach space, $J: X\rightarrow \mathbb{R}$ is a $C^1$
functional and $c\in \mathbb{R}$. We say that $\{u_n\} \subset H_0^2(\Omega)$ is
a $(PS)_c$ sequence of $J$ if
$$J(u_n)\rightarrow c \ and \ J'(u_n)\rightarrow0 \ in \ X^{-1}(\Omega) \ as \ n\rightarrow \infty,$$
where $X^{-1}$ is the dual space of $X$. Furthermore, We say that $J$ satisfies the $(PS)_c$
condition if any $(PS)_c$ sequence has a convergent subsequence.
\end{definition}

\begin{lemma}\label{Mountain Pass Lemma}(Mountain Pass Lemma \cite{Willem})
Assume that $(X,\|\cdot\|_X)$ is a real Banach space, $J:X\rightarrow \mathbb{R}$ is a $C^1$
functional and there exist $\alpha>0$ and $r>0$ such that $J$ satisfies the following mountain pass geometry:\\
(i)~$J(u)\geq\alpha>0$ if $\|u\|_X=r$;\\
(ii)~there exists a $v_0\in X$ such that $\|v_0\|_X>r$ and $J(v_0)\leq0$.\\
Then there exists a $(PS)_{c_0}$ sequence $\{u_n\}\subset X$, $i.e.$, $J(u_n)\rightarrow c_0$ and $J'(u_n)\rightarrow 0$ in $X^{-1}$ as $n\rightarrow\infty$, where $X^{-1}$ is the dual space of $X$ and
$$0<\alpha\leq c_0:=\inf_{\gamma\in \Gamma}\max_{t\in[0,1]}J(\gamma(t)),\Gamma=\{\gamma\in C([0,1],H_0^2(\Omega)):\gamma(0)=0,\gamma(1)=v_0\},$$
which is called the mountain pass level. Furthermore, $c_0$ is a critical value of $J$ if $J$ satisfies the
$(PS)_{c_0}$ condition.
\end{lemma}

To overcome the difficulties caused by the critical exponent, we recall the famous concentration compactness principle due to P. L. Lions (see \cite{Lions1} and \cite{Lions2}).
\begin{lemma}\label{jizhongjin}
Let ${u_n}\subset H_0^2(\Omega)$ be a bounded sequence with weak limit $u$, $\mu$ and $\nu$ be two nonnegative and bounded measures on $\overline{\Omega}$, such that\\
$(i)$~$|\Delta u_n|^2$ converges in the weak$^{*}$ sense of measures to a measure $\mu$,\\
$(ii)$~$|u_n|^{2^{**}}$ converges in the weak$^{*}$ sense of measures to a measure $\nu$.\\
Then for some at most countable index set $I$, points $\{x_i\}_{i\in I}\subset\overline{\Omega}$, values
$\{\mu_i\}_{i\in I}$, $\{\nu_i\}_{i\in I}\subset \mathbb{R}^{+}$, we have
\begin{equation*}
\begin{split}
(1)\ &\mu\geq |\Delta u|^2+\sum_{i\in I}\mu_i \delta_{x_i}, \ \ \mu_i>0,\\
(2)\ &\nu=|u|^{2^{**}}+\sum_{i\in I}\nu_i \delta_{x_i}, \ \ \nu_i>0,\\
(3)\ &S\nu_i^{\frac{2}{2^{**}}}\leq\mu_i,(i\in I),
\end{split}
\end{equation*}
where $\delta_{x_i}$ is the Dirac measure mass at $x_i \in \overline{\Omega}$, and $S$ is the best Sobolev constant, i.e. $S=inf\{\int_{\mathbb{R}^N}|\Delta u|^2\mathrm{d}x:\int_{\mathbb{R}^N}|u|^{2^{**}}\mathrm{d}x=1\}$.
\end{lemma}

\par
\section{The case $5\leq N\leq 7$}
\setcounter{equation}{0}

In this section, we consider problem \eqref{eq1} with $5\leq N\leq 7$ and prove Theorem \ref{th1.1}.
Throughout this section, we assume that $a\geq0, b>0$ and $\lambda>0$.
We first show that the energy functional $J(u)$ satisfies the mountain pass geometry.

\begin{lemma}\label{shanlujiegou}
Suppose that $f$ satisfies $(f1)$ and $(f2)$, then we have

$\mathrm{(i)}$ \ there exist $\alpha,r>0$ such that $J(u)\geq\alpha$ for all $u\in H_0^2(\Omega)$ with $\|u\|=r$,

$\mathrm{(ii)}$ there exists a function $v_0\in H_0^2(\Omega)$ such that $\|v_0\|>r$ and $J(v_0)\leq0$.
\end{lemma}

\begin{proof}
 $\mathrm{(i)}$ It follows from $(f1)$ and $(f2)$ that for each $\delta>0$ there exists a constant
 $C(\delta)>0$ such that
\begin{equation}\label{fu}
f(x,s)s\leq \delta s^2+C(\delta) |s|^{2^{**}}, \ \ (x,s)\in \Omega\times \mathbb{R}.
\end{equation}
Take $u\in H_0^2(\Omega)$ with $\|u\|=r$. By using $(f1)$, \eqref{fu} and Sobolev embedding inequality, we get
\begin{equation*}
\begin{split}
J(u)&=\frac{1}{2}\|u\|^2+\frac{b}{4}\|\nabla u\|_2^4-\lambda\int_{\Omega}F(x,u)\mathrm{d}x-\frac{1}{2^{**}}\int_{\Omega}|u|^{2^{**}}\mathrm{d}x\\
&\geq\frac{1}{2}\|u\|^2-\lambda\delta\int_{\Omega}u^2\mathrm{d}x-\lambda C(\delta)\int_{\Omega}|u|^{2^{**}}\mathrm{d}x-\frac{1}{2^{**}}\int_{\Omega}|u|^{2^{**}}\mathrm{d}x\\
&\geq\frac{1}{2}\|u\|^2-\lambda\delta C_1\|u\|^2-(\lambda C(\delta)+\frac{1}{2^{**}})C_2\|u\|^{2^{**}}\\
&=(\frac{1}{2}-\lambda\delta C_1){r}^2-(\lambda C(\delta)+\frac{1}{2^{**}})C_2{r}^{2^{**}},
\end{split}
\end{equation*}
where $C_1, C_2 >0$ are the Sobolev embedding constants. Fix $\delta<\frac{1}{2\lambda C_1}$.
Since $2^{**}>2$, we conclude that there exist $\alpha,r>0$ such that $J(u)\geq\alpha$ for all $u\in H_0^2(\Omega)$ with $\|u\|=r$.

$\mathrm{(ii)}$ Fix $u\in H_0^2(\Omega)\backslash \{0\}$. According to $(f1)$, we have
\begin{equation*}
\begin{split}
J(tu)=&\frac{t^2}{2}\|u\|^2+\frac{bt^4}{4}\|\nabla u\|_2^4-\lambda\int_{\Omega}F(x,tu)\mathrm{d}x-\frac{t^
{2^{**}}}{2^{**}}\int_{\Omega}|u|^{2^{**}}\mathrm{d}x\\
\leq&\frac{t^2}{2}\|u\|^2+\frac{bt^4}{4}\|u\|^4-\frac{t^{2^{**}}}{2^{**}}\int_{\Omega}|u|^{2^{**}}\mathrm{d}x.
\end{split}
\end{equation*}
Since $2^{**}>4$ when $5\leq N\leq 7$, we can find a $t_u>0$ large enough such that $J(t_u u)\leq0$ and $\|t_u u\|>r$.
Set $v_0:=t_u u$, and this completes the proof.
\end{proof}

Next we will prove that $J$ satisfies the $(PS)_c$ condition when c is smaller than some critical value.
\begin{lemma}\label{jubuPStiaojian}
Let $f$ satisfy $(f1)$ and $(f2)$, and assume that $c<(\frac{1}{2}-\frac{1}{2^{**}})S^{\frac{N}{4}}.$
Then $J$ satisfies the $(PS)_c$ condition.
\end{lemma}

\begin{proof}
The proof is divided into three steps.

{\bf Step 1.} Let $\{u_n\}\subset H_0^2(\Omega)$ be a $(PS)_c$ sequence of $J$. We first prove that $\{u_n\}$ is bounded in $H_0^2(\Omega)$.
From the definition of the $(PS)_c$ sequence, we have
\begin{equation}\label{equ3.2}
J(u_n)\rightarrow c, \ \ J'(u_n)\rightarrow 0 \ in \ H^{-2}(\Omega), \ \ \ as \ n\rightarrow\infty,
\end{equation}
which implies, by recalling $(f1)$, that
\begin{equation*}
\begin{split}
c+1+o(1)\|u_n\|&=J(u_n)-\frac{1}{\theta}\langle J'(u_n),u_n\rangle\\
&=(\frac{1}{2}-\frac{1}{\theta})\|u_n\|^2+(\frac{1}{4}-\frac{1}{\theta})b(\int_{\Omega}|\nabla u_n|^2\mathrm{d}x)^2\\
& \ \ \ +\lambda\int_{\Omega}(\frac{1}{\theta}f(x,u_n)u_n-F(x,u_n))\mathrm{d}x
+(\frac{1}{\theta}-\frac{1}{2^{**}})\int_{\Omega}|u_n|^{2^{**}}\mathrm{d}x\\
&\geq(\frac{1}{2}-\frac{1}{\theta})\|u_n\|^2.
\end{split}
\end{equation*}
Therefore, $\{u_n\}$ is bounded in $H_0^2(\Omega)$.

{\bf Step 2.} Next, we show that $\{u_n\}$ has a strongly convergent subsequence in $L^{2^{**}}(\Omega)$.
By the boundedness of $\{u_n\}$ in $H_0^2(\Omega)$ and the Sobolev embedding, one sees that there is a subsequence
of $\{u_n\}$ (which we still denote by $\{u_n\}$) such that, as $n\rightarrow\infty$,
\begin{equation}\label{un}
\begin{split}
u_n&\rightharpoonup \ u \ in \ H_0^2(\Omega),\\
u_n&\rightarrow \ u \ in \ H_0^1(\Omega),\\
u_n&\rightarrow \ u \ in \ L^p(\Omega), \ for \ all \ 1\leq p<2^{**},\\
u_n&\rightarrow \ u \ a.e. \ in \ \Omega.
\end{split}
\end{equation}
According to Lemma \ref{jizhongjin}, there exists an at most countable index set $I$ and a collection of points $\{x_i\}_{i\in I}$ such that
\begin{equation}\label{uv1}
\begin{split}
&|\Delta u_n|^2\rightharpoonup\mu\geq |\Delta u|^2+\sum_{i\in I}\mu_i \delta_{x_i}, \ \ \mu_i>0,\\
&|u_n|^{2^{**}}\rightharpoonup \nu=|u|^{2^{**}}+\sum_{i\in I}\nu_i \delta_{x_i}, \ \ \nu_i>0,
\end{split}
\end{equation}
and
\begin{equation}\label{uv}
S\nu_i^{\frac{2}{2^{**}}}\leq\mu_i,\ (i\in I).
\end{equation}
We claim that $I=\emptyset$. Suppose on the contrary that $I\neq\emptyset$ and fix $i\in I$.
Choose a function $\phi(x)\in C_0^\infty(\Omega)$ such that $\phi(x)=1$ in $B(x_i,\rho)$, $\phi(x)=0$ in
$\Omega\backslash B(x_i,2\rho)$, and $0\leq \phi(x)\leq1$ otherwise. In addition, we assume that $|\nabla\phi|\leq\frac{2}{\rho}$
and $|\Delta\phi|\leq\frac{2}{\rho^2}$ in $\Omega$. According to \eqref{equ3.2} and the boundedness
of $\{u_n\phi\}$(which follows from the boundedness of $\{u_n\}$), we have
\begin{align}\label{123}
o(1)=&\langle J'(u_n),u_n\phi\rangle\nonumber\\
=&\int_{\Omega}\Delta{u_n}\Delta(u_n \phi)\mathrm{d}x+(a+b\int_{\Omega}|\nabla u_n|^2\mathrm{d}x)\int_{\Omega}\nabla {u_n}\nabla (u_n \phi)\mathrm{d}x-\lambda\int_{\Omega}f(x,u_n)u_n \phi\mathrm{d}x\nonumber\\
&-\int_{\Omega}|u_n|^{2^{**}-2}u_n u_n \phi\mathrm{d}x\nonumber\\
=&\int_{\Omega}|\Delta u_n|^2 \phi\mathrm{d}x+2\int_{\Omega}\Delta u_n\nabla u_n \nabla\phi\mathrm{d}x
+\int_{\Omega}\Delta u_n u_n \Delta\phi\mathrm{d}x\nonumber\\
&+(a+b\int_{\Omega}|\nabla u_n|^2\mathrm{d}x)\int_{\Omega}|\nabla u_n|^2\phi\mathrm{d}x
+(a+b\int_{\Omega}|\nabla u_n|^2\mathrm{d}x)\int_{\Omega}\nabla u_n u_n \nabla\phi\mathrm{d}x\nonumber\\
&-\lambda\int_{\Omega}f(x,u_n)u_n \phi\mathrm{d}x-\int_{\Omega}|u_n|^{2^{**}}\phi\mathrm{d}x\nonumber\\
:=&I_1+I_2+I_3+I_4+I_5+I_6+I_7.
\end{align}
By the boundedness of $\{u_n\}$, the definition of $\phi$, and
recalling \eqref{un} and H\"{o}lder's inequality, we can estimate $I_2$, $I_3$ and $I_5$, respectively, by
\begin{equation}\label{ine1}
\begin{split}
\limsup\limits_{n\rightarrow\infty}\left|I_2\right|
\leq&\limsup\limits_{n\rightarrow\infty}\left(\int_{\Omega\cap B(x_i,2\rho)}|\Delta u_n|^2\mathrm{d}x\right)^{\frac{1}{2}}\left(\int_{\Omega\cap B(x_i,2\rho)}|\nabla u_n \nabla\phi|^2\mathrm{d}x\right)^{\frac{1}{2}}\\
 \leq&C\left(\int_{\Omega\cap B(x_i,2\rho)}|\nabla u \nabla\phi|^2\mathrm{d}x\right)^{\frac{1}{2}}\\
 \leq&C\left(\int_{\Omega\cap B(x_i,2\rho)}|\nabla u|^{\frac{2N}{N-2}}\mathrm{d}x\right)^{\frac{N-2}{2N}}
 \left(\int_{\Omega\cap B(x_i,2\rho)}|\nabla\phi|^N\mathrm{d}x\right)^{\frac{1}{N}}\\
 \leq&C\left(\int_{\Omega\cap B(x_i,2\rho)}|\nabla u|^{\frac{2N}{N-2}}\mathrm{d}x\right)^{\frac{N-2}{2N}} \ \ \ \rightarrow  \ 0, \ \ as \ \rho\rightarrow 0,
\end{split}
\end{equation}

\begin{equation}\label{ine2}
\begin{split}
\limsup\limits_{n\rightarrow\infty}\left|I_3\right|
\leq&\limsup\limits_{n\rightarrow\infty}\left(\int_{\Omega\cap B(x_i,2\rho)}|\Delta u_n|^2\mathrm{d}x\right)^{\frac{1}{2}}\left(\int_{\Omega\cap B(x_i,2\rho)}|u_n \Delta\phi|^2\mathrm{d}x\right)^{\frac{1}{2}}\\
 \leq&C\left(\int_{\Omega\cap B(x_i,2\rho)}|u \Delta\phi|^2\mathrm{d}x\right)^{\frac{1}{2}}\\
 \leq&C\left(\int_{\Omega\cap B(x_i,2\rho)}|u|^{2^{**}}\mathrm{d}x\right)^{\frac{1}{2^{**}}}
 \left(\int_{\Omega\cap B(x_i,2\rho)}|\Delta\phi|^{\frac{N}{2}}\mathrm{d}x\right)^{\frac{2}{N}}\\
 \leq&C\left(\int_{\Omega\cap B(x_i,2\rho)}|u|^{2^{**}}\mathrm{d}x\right)^{\frac{1}{2^{**}}} \ \ \ \rightarrow  \ 0, \ \ as \ \rho\rightarrow 0,
\end{split}
\end{equation}
and
\begin{equation}\label{ine3}
\begin{split}
\limsup\limits_{n\rightarrow\infty}\left|I_5\right|
\leq&C\limsup\limits_{n\rightarrow\infty}
\left(\int_{\Omega\cap B(x_i,2\rho)}|u_n|^2\mathrm{d}x\right)^{\frac{1}{2}}\left(\int_{\Omega\cap B(x_i,2\rho)}|\nabla u_n \nabla\phi|^2\mathrm{d}x\right)^{\frac{1}{2}}\\
 \leq&C\left(\int_{\Omega\cap B(x_i,2\rho)}|\nabla u \nabla\phi|^2\mathrm{d}x\right)^{\frac{1}{2}}\\
 \leq&C\left(\int_{\Omega\cap B(x_i,2\rho)}|\nabla u|^{\frac{2N}{N-2}}\mathrm{d}x\right)^{\frac{N-2}{2N}}
 \left(\int_{\Omega\cap B(x_i,2\rho)}|\nabla\phi|^N\mathrm{d}x\right)^{\frac{1}{N}}\\
 \leq&C\left(\int_{\Omega\cap B(x_i,2\rho)}|\nabla u|^{\frac{2N}{N-2}}\mathrm{d}x\right)^{\frac{N-2}{2N}} \ \ \ \rightarrow  \ 0, \ \ as \ \rho\rightarrow 0.
\end{split}
\end{equation}
Moreover, in view of \eqref{fu}, one has
\begin{equation}\label{ine4}
\begin{split}
\limsup\limits_{n\rightarrow\infty}\left|I_6\right|
\leq&\limsup\limits_{n\rightarrow\infty}\lambda\left|\int_{\Omega\cap B(x_i,2\rho)}f(x,u_n)u_n\phi\mathrm{d}x\right|\\
\leq&\limsup\limits_{n\rightarrow\infty}\lambda\int_{\Omega\cap B(x_i,2\rho)}\left(C(\delta) u_n^2+\delta|u_n|^{2^{**}}\right)\phi\mathrm{d}x\\
\leq&C(\delta)\lambda\int_{\Omega\cap B(x_i,2\rho)} u^2\mathrm{d}x+C\lambda\delta \ \ \rightarrow \ \ C\lambda\delta, \ \ as \ \rho\rightarrow 0,
\end{split}
\end{equation}
for all $\delta>0$. Therefore, by letting $n\rightarrow\infty$ and $\varepsilon\rightarrow 0$ successively
in \eqref{123}, and recalling \eqref{uv1} \eqref{ine1}, \eqref{ine2}, \eqref{ine3}, \eqref{ine4} and the
nonnegativity of $I_4$, one arrives at $\mu_i\leq \nu_i$.
Combing this with \eqref{uv}, we have $\nu_i\geq S\nu_i^{\frac{2}{2^{**}}}$, which implies
$$(\rm I) \ \ \nu_i=0, \ \ \ or \ \ \ (\rm {II}) \ \ \nu_i\geq S^\frac{N}{4}.$$
Next we claim that (\rm II) cannot occur. If (\rm II) holds for some $i \in I$,
according to \eqref{equ3.2} and by the use of $(f1)$, \eqref{uv1}, and \eqref{uv},
we have
\begin{equation*}
\begin{split}
c=&\lim\limits_{n\rightarrow\infty}(J(u_n)-\frac{1}{\theta}\langle J'(u_n),u_n\rangle)\\
=&\lim\limits_{n\rightarrow\infty}\bigg\{(\frac{1}{2}-\frac{1}{\theta})\int_{\Omega}|\Delta u_n|^2\mathrm{d}x+(\frac{1}{2}-\frac{1}{\theta})a\int_{\Omega}|\nabla u_n|^2\mathrm{d}x+(\frac{1}{4}-\frac{1}{\theta})b(\int_{\Omega}|\nabla u_n|^2\mathrm{d}x)^2\\
&+\lambda\int_{\Omega}(\frac{1}{\theta}f(x,u_n)u_n-F(x,u_n))\mathrm{d}x
+(\frac{1}{\theta}-\frac{1}{2^{**}})\int_{\Omega}|u_n|^{2^{**}}\mathrm{d}x\bigg\}\\
\geq&\lim\limits_{n\rightarrow\infty}(\frac{1}{2}-\frac{1}{\theta})\int_{\Omega}|\Delta u_n|^2\phi\mathrm{d}x+(\frac{1}{\theta}-\frac{1}{2^{**}})\int_{\Omega}|u_n|^{2^{**}}\phi\mathrm{d}x\\
\geq&(\frac{1}{2}-\frac{1}{\theta})\mu_i+(\frac{1}{\theta}-\frac{1}{2^{**}})\nu_i\\
\geq&(\frac{1}{2}-\frac{1}{\theta})S\nu_i^{\frac{2}{2^{**}}}+(\frac{1}{\theta}-\frac{1}{2^{**}})\nu_i\\
\geq&(\frac{1}{2}-\frac{1}{2^{**}})S^\frac{N}{4},
\end{split}
\end{equation*}
which contradicts with $c<(\frac{1}{2}-\frac{1}{2^{**}})S^{\frac{N}{4}}$. Consequently, $\nu_i=0$ for all $i\in I$. Thus we get
\begin{equation}\label{linjiexinag}
\int_{\Omega}|u_n|^{2^{**}}\mathrm{d}x\rightarrow\int_{\Omega}|u|^{2^{**}}\mathrm{d}x, \ \ as \ \ n\rightarrow \infty.
\end{equation}

{\bf Step 3.} Last, we prove that $J$ satisfies the $(PS)_c$ condition.
From $J'(u_n)\rightarrow 0$ and the boundedness of $\{u_n-u\}$, it follows that
\begin{equation}\label{zhengmingPS}
\begin{split}
o(1)=&\langle J'(u_n),u_n-u\rangle\\
=&\int_{\Omega}|\Delta u_n|^2\mathrm{d}x-\int_{\Omega}\Delta u_n\Delta u\mathrm{d}x+(a+b\int_{\Omega}|\nabla u_n|^2\mathrm{d}x)(\int_{\Omega}|\nabla u_n|^2\mathrm{d}x-\int_{\Omega}\nabla u_n\nabla u\mathrm{d}x)\\
&-\lambda\int_{\Omega}f(x,u_n)(u_n-u)\mathrm{d}x-\int_{\Omega}|u_n|^{2^{**}-2}u_n(u_n-u)\mathrm{d}x.
\end{split}
\end{equation}
By the boundedness of $\{u_n\}$, the Sobolev embedding inequality, H\"{o}lder's inequality,
and recalling \eqref{fu}, \eqref{un} and \eqref{linjiexinag}, we have
\begin{equation}\label{1234}
\begin{split}
&\limsup\limits_{n\rightarrow\infty}\left|\left(a+b\int_{\Omega}|\nabla u_n|^2\mathrm{d}x\right)\left(\int_{\Omega}|\nabla u_n|^2\mathrm{d}x-\int_{\Omega}\nabla u_n\nabla u\mathrm{d}x\right)\right|\\
\leq&\limsup\limits_{n\rightarrow\infty}C\left|\int_{\Omega}|\nabla u_n|^2\mathrm{d}x-\int_{\Omega}\nabla u_n\nabla u\mathrm{d}x\right|\\
=&0,
\end{split}
\end{equation}

\begin{equation}\label{funu}
\begin{split}
&\limsup\limits_{n\rightarrow\infty}\left|\int_{\Omega}f(x,u_n)(u_n-u)\mathrm{d}x\right|\\
\leq&\limsup\limits_{n\rightarrow\infty}C(\delta)\int_{\Omega}|u_n||u_n-u|\mathrm{d}x
+\delta\int_{\Omega}\left|u_n\right|^{2^{**}-1}|u_n-u|\mathrm{d}x\\
\leq&\limsup\limits_{n\rightarrow\infty}C(\delta)\left(\int_{\Omega}|u_n|^2\mathrm{d}x\right)^{\frac{1}{2}}
\left(\int_{\Omega}|u_n-u|^2\mathrm{d}x\right)^{\frac{1}{2}}\\
&+\delta\left(\int_{\Omega}|u_n|^{2^{**}}\mathrm{d}x\right)^{\frac{2^{**}-1}{2^{**}}}
\left(\int_{\Omega}|u_n-u|^{2^{**}}\mathrm{d}x\right)^{\frac{1}{2^{**}}}\\
=&0,
\end{split}
\end{equation}
and
\begin{equation}\label{un-u}
\begin{split}
&\limsup\limits_{n\rightarrow\infty}\left|\int_{\Omega}|u_n|^{2^{**}-2}u_n(u_n-u)\mathrm{d}x\right|\\
\leq&\limsup\limits_{n\rightarrow\infty}\left(\int_{\Omega}|u_n|^{2^{**}}\mathrm{d}x\right)^{\frac{2^{**}-1}{2^{**}}}
\left(\int_{\Omega}|u_n-u|^{2^{**}}\mathrm{d}x\right)^{\frac{1}{2^{**}}}\\
\leq&\limsup\limits_{n\rightarrow\infty}C\left(\int_{\Omega}\left|u_n-u\right|^{2^{**}}\mathrm{d}x\right)^{\frac{1}{2^{**}}}\\
=&0.
\end{split}
\end{equation}
Therefore, letting $n\rightarrow\infty$ in \eqref{zhengmingPS} and
using \eqref{1234}, \eqref{funu} and \eqref{un-u}, we get
$$\int_{\Omega}|\Delta u_n|^2\mathrm{d}x-\int_{\Omega}\Delta u_n\Delta u\mathrm{d}x=0, \ \ as \ n\rightarrow\infty,$$
which, together with \eqref{un}, ensures that $\|u_n-u\|\rightarrow0$ as $n\rightarrow\infty$. This completes the proof.
\end{proof}

The successful application of the Mountain Pass Lemma in proving the existence of weak
solutions for $5\leq N\leq 7$ is based on the following lemma.

\begin{lemma}\label{lemma3.3}
Assume that $5\leq N\leq 7$. If there exists a function $u_0\in H_0^2(\Omega)\setminus\{0\}$ such that
\begin{equation}\label{supj}
\sup_{t\geq0}J(tu_0)<(\frac{1}{2}-\frac{1}{2^{**}})S^\frac{N}{4},
\end{equation}
then problem \eqref{eq1} possesses at least one nontrivial solution for all $\lambda>0$.
\end{lemma}

\begin{proof}
From Lemma \ref{shanlujiegou}, we know that $J$ has a mountain pass geometry around $0$.
According to the Mountain Pass Lemma, there exists a sequence $\{u_n\}\subset H_0^2(\Omega)$
such that $J(u_n)\rightarrow c_0$ and $J'(u_n)\rightarrow 0$ in $H^{-2}(\Omega)$ as $n\rightarrow \infty$,
where
$$0<\alpha\leq c_0:=\inf_{\gamma\in \Gamma}\max_{t\in[0,1]}J(\gamma(t)),\Gamma=\{\gamma\in C([0,1],H_0^2(\Omega)):\gamma(0)=0,\gamma(1)=t_{u_0}u_0\},$$
and $t_{u_0}>0$ is the constant determined in the proof of Lemma \ref{shanlujiegou}.
Moreover, it follows from \eqref{supj} that
$$c_0=\inf_{\gamma\in \Gamma}\max_{t\in [0,1]}J(\gamma(t))
\leq\max_{t\in [0,1]}J(tt_{u_0}u_0)\leq\sup_{t\geq0}J(tu_0)<(\frac{1}{2}-\frac{1}{2^{**}})S^\frac{N}{4}.$$
By virtue of Lemma \ref{jubuPStiaojian}, we know that $J$ satisfies the $(PS)_{c_0}$ condition,
and there exist a function $u\in H_0^2(\Omega)$ and a convergent subsequence of $\{u_n\}$,
still denoted by $\{u_n\}$, such that $u_n\rightarrow u$ in $H_0^2(\Omega)$ as $n\rightarrow \infty$.
Therefore, $J(u)=c_0$ and $J'(u)=0$, i.e.,
$u$ is a mountain pass type (and of course nontrivial) solution to problem \eqref{eq1}.
This completes the proof.
\end{proof}

To complete the proof of Theorem \ref{th1.1}, it remains to show that \eqref{supj}
is valid for some $u_0\in H_0^2(\Omega)$. Inspired by \cite{DD1990}, \cite{LHW2023} and \cite{RP1993},
we shall do this with the help of the truncated Talenti's function. For any $\varepsilon>0$, define
$$u_{\varepsilon}(x)=\frac{[N(N-4)(N^2-4)\varepsilon^4]^{\frac{N-4}{8}}}{(\varepsilon^2+|x|^2)^{\frac{N-4}{2}}}, \ \ x\in \mathbb{R}^N.$$
It is well known that $u_\varepsilon$ is a solution to the following critical problem
$$\Delta^2 u=u^{2^{**}-1}, \ \ x\in \ \mathbb{R}^N, \ N\geq5,$$
and $\|\Delta u_\varepsilon\|_2^2=\|u_\varepsilon\|_{2^{**}}^{2^{**}}=S^{\frac{N}{4}}$,
where $S>0$ is the constant given in Lemma \ref{jizhongjin} which can also be characterized as $S=\inf\limits_{u\in H_0^2(\Omega)\backslash\{0\}}\dfrac{\|\Delta u\|_2^2}{\|u\|_{2^{**}}^{2}}
=\dfrac{\|\Delta u_\varepsilon\|_2^2}{\|u_\varepsilon\|_{2^{**}}^{2}}$.

With no loss of generality, we may assume that $0\in \omega\subset\Omega$, where $\omega$ is given in $(f3)$.
For any $\varepsilon>0$, set $\psi_\varepsilon(x)=\varphi(x)u_\varepsilon(x)$,
where $\varphi(x)=\varphi(|x|)\in C^2(\overline{\Omega},\mathbb{R})$ is a given function defined in $\Omega$
with $supp\varphi\subset B(0,\varrho)\subset \Omega$ and
$$\varphi(0)=1, \ \varphi(\varrho)=\varphi'(\varrho)=0.$$
According to \cite{LvZongyan2021}, we have the following estimate:
\begin{equation}\label{es1}
\ \ \ \ \ \ \ \ \begin{cases}
\|\Delta\psi_\varepsilon\|_2^2=S^{\frac{N}{4}}+O(\varepsilon^{N-4}),\\
\|\nabla\psi_\varepsilon\|_2^2=O(\varepsilon),\\
\|\psi_\varepsilon\|_2^2=O(\varepsilon^{N-4}),\\
\|\psi_\varepsilon\|_{2^{**}}^{2^{**}}=S^{\frac{N}{4}}+O(\varepsilon^{c_N}),
\end{cases}
\end{equation}
where $c_N=1.81$ when $N=5$, $c_N=2.6$ when $N=6$, and $c_N=3.53$ when $N=7$.

\begin{lemma}\label{le3.4}
Assume that $5\leq N\leq 7$ and $f$ satisfies $(f1)$-$(f3)$. Then there exists
a constant $\varepsilon_1>0$ such that
\begin{equation}\label{level}
\sup_{t\geq0}J(t\psi_\varepsilon)<(\frac{1}{2}-\frac{1}{2^{**}})S^\frac{N}{4},
\end{equation}
for all $\varepsilon\in(0,\varepsilon_1)$.
\end{lemma}

\begin{proof}
Consider the fibering map $j(t):(0,+\infty)\rightarrow \mathbb{R}$ defined by
$$j(t):=J(t\psi_\varepsilon)=\frac{t^2}{2}{\|\psi_\varepsilon\|}^2+\frac{bt^4}{4}\left(\int_{\Omega}|\nabla \psi_\varepsilon|^2\mathrm{d}x\right)^2-\lambda\int_{\Omega}F(x,t\psi_\varepsilon)\mathrm{d}x
-\frac{t^{2^{**}}}{2^{**}}\int_{\Omega}|\psi_\varepsilon|^{2^{**}}\mathrm{d}x.$$
In view of $(f1)$, \eqref{fu} and the fact that $2^{**}>4$, we have $\lim\limits_{t\rightarrow 0}j(t)=0$, $\lim\limits_{t\rightarrow +\infty}j(t)=-\infty$ and $j(t)>0$ for $t>0$ suitably small. Therefore, there exists a $t_\varepsilon>0$ such that $$j(t_\varepsilon)=\sup_{t\geq0}j(t),$$
and
\begin{equation}\label{df}
\|\psi_\varepsilon\|^2+b{t_\varepsilon}^2\left(\int_{\Omega}|\nabla \psi_\varepsilon|^2\mathrm{d}x\right)^2-
\frac{\lambda}{t_\varepsilon}\int_{\Omega}f(x,t_\varepsilon\psi_\varepsilon)\psi_\varepsilon\mathrm{d}x
-{t_\varepsilon}^{2^{**}-2}\int_{\Omega}|\psi_\varepsilon|^{2^{**}}\mathrm{d}x=0.
\end{equation}

We claim that there exist two positive constants $C_1$, $C_2$ such that
\begin{equation}\label{tisbounded}
0<C_1\leq t_\varepsilon \leq C_2<+\infty,
\end{equation}
uniformly for suitably small $\varepsilon$. Indeed, from \eqref{df} and $(f1)$ we have
\begin{equation}\label{1}
\|\psi_\varepsilon\|^2+b{t_\varepsilon}^2(\int_{\Omega}|\nabla \psi_\varepsilon|^2\mathrm{d}x)^2
\geq{t_\varepsilon}^{2^{**}-2}\int_{\Omega}|\psi_\varepsilon|^{2^{**}}\mathrm{d}x,
\end{equation}
which, together with \eqref{es1}, implies that $t_\varepsilon$ is bounded from the above.
As for the lower bound of $t_\varepsilon$, we use \eqref{fu}, \eqref{df} and Sobolev embedding inequality to obtain
\begin{equation}\label{2}
(1-\lambda\delta C)\|\psi_\varepsilon\|^2+b{t_\varepsilon}^2\left(\int_{\Omega}|\nabla \psi_\varepsilon|^2\mathrm{d}x\right)^2
\leq (1+\lambda C(\delta)){t_\varepsilon}^{2^{**}-2}\int_{\Omega}|\psi_\varepsilon|^{2^{**}}\mathrm{d}x,\\
\end{equation}
where $C>0$ is the embedding constant and $\delta<\frac{1}{\lambda C}$.
The lower bound of $t_\varepsilon$ follows from \eqref{es1} and \eqref{2}.

Next, we prove
\begin{equation}\label{limit t(varepsilon)}
\lim\limits_{\varepsilon\rightarrow 0}t_\varepsilon=1,
\end{equation}
which is easy to verify once we can show that
\begin{equation}\label{ft}
\lim\limits_{\varepsilon\rightarrow 0}\int_{\Omega}\frac{f(x,t_\varepsilon\psi_\varepsilon)\psi_\varepsilon}{t_\varepsilon}\mathrm{d}x=0.
\end{equation}
Indeed, if \eqref{ft} is valid, by letting $\varepsilon\rightarrow 0$ in \eqref{df} and
recalling \eqref{es1} one obtains \eqref{limit t(varepsilon)}. To show \eqref{ft},
we use \eqref{fu}, \eqref{es1} and \eqref{tisbounded} to obtain
\begin{equation*}
\begin{split}
&\limsup\limits_{\varepsilon\rightarrow 0}\left|\int_{\Omega}\frac{f(x,t_\varepsilon\psi_\varepsilon)\psi_\varepsilon}{t_\varepsilon}\mathrm{d}x\right|\\
\leq&\limsup\limits_{\varepsilon\rightarrow 0}\bigg(C(\delta)\int_{\Omega}\psi_\varepsilon^2\mathrm{d}x
+\delta t_\varepsilon^{2^{**}-2}\int_{\Omega}\psi_\varepsilon^{2^{**}}\mathrm{d}x\bigg)\\
\leq&\limsup\limits_{\varepsilon\rightarrow 0}\bigg(C(\delta)\int_{\Omega}\psi_\varepsilon^2\mathrm{d}x
+C\delta\int_{\Omega}\psi_\varepsilon^{2^{**}}\mathrm{d}x\bigg)\\
\leq&C S^{\frac{N}{4}}\delta
\end{split}
\end{equation*}
for all $\delta>0$. By the arbitrariness of $\delta>0$, \eqref{ft} follows.

Set $Q(t)=\dfrac{1}{2}t^2-\dfrac{1}{2^{**}}t^{2^{**}}$. Direct calculation shows
that $Q(t)\leq Q(1)=\dfrac{2}{N}$ for all $t>0$.
According to \eqref{es1}, we can draw the following conclusions.

For $N=5,6$,
\begin{equation}\label{level 5-6}
\begin{split}
&\sup_{t\geq0}J(t\psi_\varepsilon)=J(t_\varepsilon\psi_\varepsilon)\\
=&\frac{t_\varepsilon^2}{2}\int_{\Omega}|\Delta \psi_\varepsilon|^2\mathrm{d}x+\frac{at_\varepsilon^2}{2}\int_{\Omega}|\nabla \psi_\varepsilon|^2\mathrm{d}x+\frac{bt_\varepsilon^4}{4}(\int_{\Omega}|\nabla \psi_\varepsilon|^2\mathrm{d}x)^2-\lambda\int_{\Omega}F(x,t_\varepsilon\psi_\varepsilon)\mathrm{d}x\\
&-\frac{t_\varepsilon^{2^{**}}}{2^{**}}\int_{\Omega}|\psi_\varepsilon|^{2^{**}}\mathrm{d}x\\
\leq&(\frac{1}{2}-\frac{1}{2^{**}})S^{\frac{N}{4}}+O(\varepsilon)
-\lambda\int_{\Omega}F(x,t_\varepsilon\psi_\varepsilon)\mathrm{d}x.
\end{split}
\end{equation}

For $N=7$, if $a=0$,
\begin{equation}\label{level 7-1}
\begin{split}
&\sup_{t\geq0}J(t\psi_\varepsilon)=J(t_\varepsilon\psi_\varepsilon)\\
\leq&\frac{t_\varepsilon^2}{2}\int_{\Omega}|\Delta \psi_\varepsilon|^2\mathrm{d}x+\frac{bt_\varepsilon^4}{4}(\int_{\Omega}|\nabla \psi_\varepsilon|^2\mathrm{d}x)^2-\lambda\int_{\Omega}F(x,t_\varepsilon\psi_\varepsilon)\mathrm{d}x
-\frac{t_\varepsilon^{2^{**}}}{2^{**}}\int_{\Omega}|\psi_\varepsilon|^{2^{**}}\mathrm{d}x\\
\leq&(\frac{1}{2}-\frac{1}{2^{**}})S^{\frac{N}{4}}+O(\varepsilon^2)
-\lambda\int_{\Omega}F(x,t_\varepsilon\psi_\varepsilon)\mathrm{d}x.
\end{split}
\end{equation}

For $N=7$, if $a>0$,
\begin{equation}\label{level 7-2}
\begin{split}
&\sup_{t\geq0}J(t\psi_\varepsilon)=J(t_\varepsilon\psi_\varepsilon)\\
\leq&\frac{t_\varepsilon^2}{2}\int_{\Omega}|\Delta \psi_\varepsilon|^2\mathrm{d}x+\frac{at_\varepsilon^2}{2}\int_{\Omega}|\nabla \psi_\varepsilon|^2\mathrm{d}x+\frac{bt_\varepsilon^4}{4}(\int_{\Omega}|\nabla \psi_\varepsilon|^2\mathrm{d}x)^2-\lambda\int_{\Omega}F(x,t_\varepsilon\psi_\varepsilon)\mathrm{d}x\\
&-\frac{t_\varepsilon^{2^{**}}}{2^{**}}\int_{\Omega}|\psi_\varepsilon|^{2^{**}}\mathrm{d}x\\
\leq&(\frac{1}{2}-\frac{1}{2^{**}})S^{\frac{N}{4}}+O(\varepsilon)
-\lambda\int_{\Omega}F(x,t_\varepsilon\psi_\varepsilon)\mathrm{d}x.
\end{split}
\end{equation}
It is easily seen from \eqref{level 5-6}-\eqref{level 7-2} that \eqref{level} is true if we can prove
\begin{equation}\label{Fe}
\lim\limits_{\varepsilon\rightarrow 0}\varepsilon^{-1}\int_{\Omega}F(x,t_\varepsilon\psi_\varepsilon)\mathrm{d}x=+\infty.
\end{equation}

It remains to show that \eqref{Fe} is valid to complete the proof.
The argument is almost the same as that in the proof of Lemma \ref{Mountain Pass Lemma} in \cite{BreNi},
and we only sketch the outline here for the convenience of the reader.
Set $\underline{F}(s):=\inf_{x\in \omega} F(x,s)$. According to $(f1)$,
$F(x,s)$ is nondecreasing in $s\geq0$ for each $x\in\Omega$ and so is $\underline{F}(s)$ the definition.
Therefore,
\begin{equation*}
\begin{split}
&\varepsilon^{-1}\int_{\Omega}F(x,t_\varepsilon\psi_\varepsilon)\mathrm{d}x
\geq\varepsilon^{-1}\int_{\omega}\underline{F}(t_\varepsilon\psi_\varepsilon)\mathrm{d}x\\
=&\varepsilon^{-1}\int_{\omega}\underline{F}\left(t_\varepsilon\varphi(x) \frac{[N(N-4)(N^2-4)]^{\frac{N-4}{8}}\varepsilon^{\frac{N-4}{2}}}{(\varepsilon^2+|x|^2)^\frac{N-4}{2}}\right)\mathrm{d}x\\
\geq&\varepsilon^{-1}\int_{B_{r_0}}\underline{F}\left(t_\varepsilon C_0 \frac{[N(N-4)(N^2-4)]^{\frac{N-4}{8}}\varepsilon^{\frac{N-4}{2}}}{(\varepsilon^2+|x|^2)^\frac{N-4}{2}}\right)\mathrm{d}x\\
:=&\varepsilon^{-1}\int_{B_{r_0}}\underline{F}\left(C(N,\varepsilon) \frac{\varepsilon^{\frac{N-4}{2}}}{(\varepsilon^2+|x|^2)^\frac{N-4}{2}}\right)\mathrm{d}x\\
=&\omega_N\varepsilon^{-1}\int_{0}^{r_0} \underline{F}\left(C(N,\varepsilon) \frac{\varepsilon^{\frac{N-4}{2}}}{(\varepsilon^2+s^2)^\frac{N-4}{2}}\right)s^{N-1}\mathrm{d}s\\
=&\omega_N\varepsilon^{N-1}\int_{0}^{\frac{r_0}{\varepsilon}} \underline{F}\left(C(N,\varepsilon) \frac{\varepsilon^{\frac{4-N}{2}}}{(1+t^2)^\frac{N-4}{2}}\right)t^{N-1}\mathrm{d}t,
\end{split}
\end{equation*}
where $C(N,\varepsilon):=t_\varepsilon C_0 [N(N-4)(N^2-4)]^{\frac{N-4}{8}}(N=5,6,7)$,
$C_0$ is a positive constant and $r_0$ is suitably small such that $\varphi(x)\geq C_0$ in $B_{r_0}$.
In view of \eqref{tisbounded}, we know that there exist two positive constants $C_3$, $C_4$ such that
\begin{equation}\label{cnisbounded}
0<C_3\leq C(N,\varepsilon) \leq C_4<+\infty.
\end{equation}
By the monotonicity of $\underline{F}(s)$ and \eqref{cnisbounded}, one has
\begin{equation*}
\begin{split}
&\omega_N\varepsilon^{N-1}\int_{0}^{\frac{r_0}{\varepsilon}} \underline{F}\left(C(N,\varepsilon) \frac{\varepsilon^{\frac{4-N}{2}}}{(1+t^2)^\frac{N-4}{2}}\right)t^{N-1}\mathrm{d}t\\
\geq&\omega_N\varepsilon^{N-1}\int_{0}^{\frac{r_0}{\varepsilon}} \underline{F}\left(C_3 \frac{\varepsilon^{\frac{4-N}{2}}}{(1+t^2)^\frac{N-4}{2}}\right)t^{N-1}\mathrm{d}t.
\end{split}
\end{equation*}

By $(f3)$ and the definition of $\underline{F}$, for each $M>0$,
there exists a constant $s_0>0$ such that $\underline{F}(s)\geq Ms^{N^2-14N+53}$ for all $s\geq s_0$.
Notice that there exists a constant $C_5>0$, independent of $\varepsilon>0$, such that
$$\frac{C_3\varepsilon^{\frac{4-N}{2}}}{(1+t^2)^\frac{N-4}{2}}\geq s_0,$$
for all $t\leq C_5\varepsilon^{-\frac{1}{2}}\leq r_0 \varepsilon^{-1}$ if $\varepsilon>0$ is sufficiently small. Using this fact, we get
\begin{equation*}
\begin{split}
&\omega_N\varepsilon^{N-1}\int_{0}^{\frac{r_0}{\varepsilon}} \underline{F}\left(C_3 \frac{\varepsilon^{\frac{4-N}{2}}}{(1+t^2)^\frac{N-4}{2}}\right)t^{N-1}\mathrm{d}t\\
 \geq&\omega_N\varepsilon^{N-1}\int_{0}^{\frac{C_5}{\sqrt{\varepsilon}}}
 M\left(\frac{C_3\varepsilon^{\frac{4-N}{2}}}{(1+t^2)^\frac{N-4}{2}}\right)^{N^2-14N+53}t^{N-1}\mathrm{d}t.
\end{split}
\end{equation*}
When $N=5,6,7$, we get
\begin{equation*}
\begin{split}
&\omega_N\varepsilon^{N-1}\int_{0}^{\frac{C_5}{\sqrt{\varepsilon}}}
 M\left(\frac{C_3\varepsilon^{\frac{4-N}{2}}}{(1+t^2)^\frac{N-4}{2}}\right)^{N^2-14N+53}t^{N-1}\mathrm{d}t\\
=&\omega_NM(C_3)^{N^2-14N+53}\int_{0}^{\frac{C_5}{\sqrt{\varepsilon}}}\frac{t^{N-1}}{(1+t^2)^{N-1}}\mathrm{d}t\\
\geq&CM, \ \ \ \forall \ M>0.
\end{split}
\end{equation*}
Therefore, we get \eqref{Fe}. This completes the proof.
\end{proof}

{\bf Proof of Theorem \ref{th1.1}.} When $5\leq N\leq 7$, assume that $f$ satisfies $(f1)$-$(f3)$.
It follows from Lemma \ref{le3.4} that condition \eqref{supj} holds. Consequently,
problem \eqref{eq1} possesses at least one nontrivial solution for all $\lambda>0$
by Lemma \ref{lemma3.3}. This completes the proof.

\par
\section{The case $N\geq 8$}
\setcounter{equation}{0}

In this section, we consider problem \eqref{eq1} with $N\geq8$. It is worthy pointing out
that the method used in Section 3 is not applicable to this case since the boundedness of
the (PS) sequence can not be obtained in general. Inspired by \cite{Figueiredo2013}, we introduce
an appropriate truncation on the nonlocal coefficient to overcome this difficulty.

Assume that $a>0$, $b>0$ and $\lambda>0$. Let $\psi(t)=a+bt$.
Then $\psi(t):[0,+\infty)\rightarrow \mathbb{R}^+$ is strictly increasing, unbounded
and $\psi(t)\geq a$ for all $t\geq0$. So given $c_0\in \mathbb{R}$ such that
$a<c_0<\frac{\theta}{2}a$, there exists a unique $t_0 > 0$ such that $\psi(t_0)=c_0$.
Set
\begin{eqnarray*}
\ \ \ \ \ \ \ \ \psi_0(t)=\begin{cases}
\psi(t),&if \ 0\leq t\leq t_0,\\
c_0,&if \ t\geq t_0.
\end{cases}
\end{eqnarray*}
Obviously, we can get
\begin{equation}\label{ine4.1}
a\leq\psi_0(t)\leq c_0, \qquad t\geq0.
\end{equation}

The proof of Theorem \ref{th1.2} is based on a careful study of the following auxiliary problem:
\begin{equation}\label{eq2}
\ \ \ \ \ \ \ \ \begin{cases}
\Delta^2u-\psi_0(\int_{\Omega}|\nabla{u}|^2\mathrm{d}x) \Delta{u}=\lambda{f(x,u)}+|u|^{2^{**}-2}u,&x\in\Omega,\\
u=\dfrac{\partial u}{\partial\overrightarrow{n}}=0,&x\in\partial\Omega.
\end{cases}
\end{equation}
The energy functional associated with problem \eqref{eq2} is given by
$$J_\psi(u)=\frac{1}{2}\int_{\Omega}|\Delta u|^2\mathrm{d}x+\frac{1}{2}\widehat{\psi_0}(\int_{\Omega}|\nabla u|^2\mathrm{d}x)
-\lambda\int_{\Omega}F(x,u)\mathrm{d}x-\frac{1}{2^{**}}\int_{\Omega}|u|^{2^{**}}\mathrm{d}x,$$
where $\widehat{\psi_0}(t)=\int_0^t\psi_0(s)\mathrm{d}s$ satisfies
\begin{equation}\label{88}
 at\leq\widehat{\psi_0}(t)\leq c_0t, \ \ t\geq0.
\end{equation}
Moreover, the derivative operator of $J_\psi$ is given by
\begin{equation*}
\begin{split}
\langle J_\psi'(u),v\rangle=&\int_{\Omega}\Delta u\Delta v\mathrm{d}x+\psi_0(\int_{\Omega}|\nabla u|^2\mathrm{d}x)\int_{\Omega}\nabla u\nabla v\mathrm{d}x
-\lambda\int_{\Omega}f(x,u)v\mathrm{d}x-\int_{\Omega}|u|^{2^{**}-2}uv\mathrm{d}x,
\end{split}
\end{equation*}
for all $u,\ v\in H_0^2(\Omega)$.

It is easily seen that nontrivial critical point $u$ of $J_\psi(\cdot)$ corresponds to
nontrivial solution to problem \eqref{eq2}. Furthermore, if $\|\nabla u\|_2^2<t_0$,
$u$ is also a solution to the original problem \eqref{eq1}. We first show the following auxiliary result.

\begin{theorem}\label{th4.1}
For the case $N\geq8$, let $a>0, b>0$, and assume that $f$ satisfies $(f1)$, $(f2)$ and $(f4)$.
Then there exists a constant $\lambda_0\geq0$ such that problem \eqref{eq2} has at least one
nontrivial solution for all $\lambda>\lambda_0$.
\end{theorem}

Theorem \ref{th4.1} is proved by a series of lemmas. We first show that the energy functional
$J_\psi(u)$ satisfies the mountain pass geometry.

\begin{lemma}\label{le4.1}
Suppose that $f$ satisfies $(f1)$ and $(f2)$. Then we have

$\mathrm{(i)}$ \ there exist $\alpha,r>0$ such that $J_\psi(u)\geq\alpha$ for all $u\in H_0^2(\Omega)$ with $\|u\|=r$.

$\mathrm{(ii)}$ there exists a function $v_0\in H_0^2(\Omega)$ such that $\|v_0\|> r$ and $J_\psi(v_0)\leq0$.
\end{lemma}
\begin{proof}
 $\mathrm{(i)}$ Take $u\in H_0^2(\Omega)$ with $\|u\|=r$. Using \eqref{fu}, \eqref{88} and Sobolev embedding inequality, we get
\begin{equation*}
\begin{split}
J_\psi(u)=&\frac{1}{2}\int_{\Omega}|\Delta u|^2\mathrm{d}x+\frac{1}{2}\widehat{\psi_0}(\int_{\Omega}|\nabla u|^2\mathrm{d}x)-\lambda\int_{\Omega}F(x,u)\mathrm{d}x-\frac{1}{2^{**}}\int_{\Omega}|u|^{2^{**}}\mathrm{d}x\\
\geq&\frac{1}{2}\int_{\Omega}|\Delta u|^2\mathrm{d}x+\frac{1}{2}a\int_{\Omega}|\nabla u|^2\mathrm{d}x-\lambda\delta\int_{\Omega}u^2\mathrm{d}x-\lambda C(\delta)\int_{\Omega}|u|^{2^{**}}\mathrm{d}x-\frac{1}{2^{**}}\int_{\Omega}|u|^{2^{**}}\mathrm{d}x\\
\geq&\frac{1}{2}\|u\|^2-\lambda\delta C_1\|u\|^2-(\lambda C(\delta)+\frac{1}{2^{**}})C_2\|u\|^{2^{**}}\\
=&(\frac{1}{2}-\lambda\delta C_1)r^2-(\lambda C(\delta)+\frac{1}{2^{**}})C_2r^{2^{**}},
\end{split}
\end{equation*}
where $C_1, C_2 >0$ are the Sobolev embedding constants. Taking $\delta<\frac{1}{2\lambda C_1}$ to
conclude that there exist $\alpha,r>0$ such that $J_\psi(u)\geq\alpha$ for all $u\in H_0^2(\Omega)$ with $\|u\|=r$.

$\mathrm{(ii)}$  Fix $u\in H_0^2(\Omega)\backslash \{0\}$. According to $(f1)$ and \eqref{88}, we have
\begin{equation*}
\begin{split}
J_\psi(tu)=&\frac{t^2}{2}\int_{\Omega}|\Delta u|^2\mathrm{d}x+\frac{1}{2}\widehat{\psi_0}(t^2\int_{\Omega}|\nabla u|^2\mathrm{d}x)-\lambda\int_{\Omega}F(x,tu)\mathrm{d}x-\frac{t^
{2^{**}}}{2^{**}}\int_{\Omega}|u|^{2^{**}}\mathrm{d}x.\\
\leq&\frac{t^2}{2}\int_{\Omega}|\Delta u|^2\mathrm{d}x+\frac{t^2}{2}c_0\int_{\Omega}|\nabla u|^2\mathrm{d}x-\frac{t^
{2^{**}}}{2^{**}}\int_{\Omega}|u|^{2^{**}}\mathrm{d}x.
\end{split}
\end{equation*}
Since $2^{**}>2$, we can find $t_u>0$ suitably large such that $J_\psi(t_u u)\leq0$ and $\|t_u u\|>r$.
Set $v_0:=t_uu$. This completes the proof.
\end{proof}

Now we define $$c_\lambda:=\inf_{\gamma\in \Gamma}\max_{t\in [0,1]}J_\psi(\gamma(t)),\Gamma=\{\gamma\in C([0,1],H_0^2(\Omega)):\gamma(0)=0,\gamma(1)=v_0\}.$$
We will prove that $c_\lambda\rightarrow 0$ as $\lambda\rightarrow\infty$, which plays an important role
in proving that $J_\psi$ satisfies the $(PS)_{c_\lambda}$ condition.

\begin{lemma}\label{le4.2}
Assume that $f$ satisfies $(f1)$, $(f2)$ and $(f4)$. Then $c_\lambda\rightarrow 0$ as $\lambda\rightarrow\infty$.
\end{lemma}

\begin{proof}
Take a nonnegative function $\widetilde{u}\in C_0^\infty(\Omega)$ satisfying $\|\Delta \widetilde{u}\|_2=1$,
$\widetilde{u}(x)>0$ for $x\in \widetilde{\omega}$ and $\widetilde{u}(x)=0$ for $x\in\Omega\setminus\widetilde{\omega}$,
and consider the fibering map $j_\psi(t):(0,+\infty)\rightarrow \mathbb{R}$ defined by
\begin{equation}\label{JTtu}
\begin{split}
j_\psi(t):=J_\psi(t\widetilde{u})=&\frac{t^2}{2}\int_{\Omega}|\Delta \widetilde{u}|^2\mathrm{d}x+\frac{1}{2}\widehat{\psi_0}(t^2\int_{\Omega}|\nabla \widetilde{u}|^2\mathrm{d}x)-\lambda\int_{\Omega}F(x,t\widetilde{u})\mathrm{d}x
-\frac{t^{2^{**}}}{2^{**}}\int_{\Omega}|\widetilde{u}|^{2^{**}}\mathrm{d}x\\
=&\frac{t^2}{2}+\frac{1}{2}\widehat{\psi_0}(t^2\int_{\Omega}|\nabla \widetilde{u}|^2\mathrm{d}x)-\lambda\int_{\Omega}F(x,t\widetilde{u})\mathrm{d}x
-\frac{t^{2^{**}}}{2^{**}}\int_{\Omega}|\widetilde{u}|^{2^{**}}\mathrm{d}x.
\end{split}
\end{equation}
Using \eqref{88}, we have
\begin{equation}\label{JTtu}
\begin{split}
j_\psi(t)\leq\frac{t^2}{2}+\frac{t^2}{2}c_0\int_{\Omega}|\nabla \widetilde{u}|^2\mathrm{d}x-\lambda\int_{\widetilde{\omega}}F(x,t\widetilde{u})\mathrm{d}x
-\frac{t^{2^{**}}}{2^{**}}\int_{\Omega}|\widetilde{u}|^{2^{**}}\mathrm{d}x.
\end{split}
\end{equation}
Since $2^{**}>2$, we have $\lim\limits_{t\rightarrow0}j_\psi(t)=0$ and $\lim\limits_{t\rightarrow +\infty}j_\psi(t)=-\infty$.
Moreover, recalling \eqref{fu}, \eqref{88} and using the similar argument to that in the proof of Lemma \ref{le3.4},
one sees that $j_\psi(t)$ is positive for $t > 0$ suitably small. Therefore, there exists a $t_\lambda>0$
such that $j_\psi(t_\lambda)=\sup_{t\geq0}j_\psi(t)$ and
\begin{equation}\label{df2}
1+\psi_0(t_\lambda^2\int_{\Omega}|\nabla \widetilde{u}|^2\mathrm{d}x)\int_{\Omega}|\nabla \widetilde{u}|^2\mathrm{d}x-\frac{\lambda}{t_\lambda}\int_{\omega}f(x,t_\lambda \widetilde{u})\widetilde{u}\mathrm{d}x
-t_\lambda^{2^{**}-2}\int_{\Omega}|\widetilde{u}|^{2^{**}}\mathrm{d}x=0.
\end{equation}
According to $(f1)$, we have
$$1+\psi_0(t_\lambda^2\int_{\Omega}|\nabla \widetilde{u}|^2\mathrm{d}x)\int_{\Omega}|\nabla \widetilde{u}|^2\mathrm{d}x
\geq t_\lambda^{2^{**}-2}\int_{\Omega}|\widetilde{u}|^{2^{**}}\mathrm{d}x,$$
which implies that $t_\lambda$ is bounded since $a\leq\psi_0\leq c_0$ and $2^{**}>2$.

We claim that $t_\lambda\rightarrow 0$ as $\lambda\rightarrow\infty$. If not, there would exist a sequence
$\{\lambda_n\}$ and a constant $\beta>0$ such that $\lambda_n\rightarrow\infty$ and $t_{\lambda_n}\rightarrow\beta$ as $n\rightarrow\infty$.
In view of $(f1)$, $(f2)$, $(f4)$ and the definition of $\widetilde{u}$, we have
$$\frac{1}{t_{\lambda_n}}\int_{\widetilde{\omega}}f(x,t_{\lambda_n} \widetilde{u})\widetilde{u}\mathrm{d}x\rightarrow
 \frac{1}{\beta}\int_{\widetilde{\omega}}f(x,\beta \widetilde{u})\widetilde{u}\mathrm{d}x>0, \ as \ n\rightarrow\infty.$$
On the other hand, by taking $\lambda=\lambda_n$ and letting $n\rightarrow\infty$ in \eqref{df2} it follows
$$\int_{\widetilde{\omega}}f(x,\beta \widetilde{u})\widetilde{u}\mathrm{d}x=0,$$
a contradiction. Therefore, $\lim\limits_{\lambda\rightarrow\infty}t_\lambda=0$ as claimed.

In view of $(f1)$ and \eqref{JTtu}, we have
$$0<\alpha\leq c_\lambda=\inf_{\gamma\in \Gamma}\max_{t\in [0,1]}J_\psi(\gamma(t))
\leq\max_{t\in [0,1]}J_\psi(tt_{\widetilde{u}}\widetilde{u})\leq\sup_{t\geq0}J_\psi(t\widetilde{u})\leq\frac{t_\lambda^2}{2}
+\frac{t_\lambda^2}{2}c_0\int_{\Omega}|\nabla \widetilde{u}|^2\mathrm{d}x,$$
where $t_{\widetilde{u}}>0$ is the constant determined in the proof of Lemma \ref{le4.1}.
Since $t_\lambda\rightarrow 0$ as $\lambda\rightarrow\infty$,
we obtain $c_\lambda \rightarrow 0$ as $\lambda \rightarrow \infty$. The proof is complete.
\end{proof}

\begin{lemma}\label{le4.3}
Let $f$ satisfy $(f1)$ and $(f2)$, and assume that
$c<(\frac{1}{2}-\frac{1}{2^{**}})S^{\frac{N}{4}}.$
Then $J_\psi$ satisfies the $(PS)_c$ condition.
\end{lemma}

\begin{proof}
We still divide the proof into three steps.

{\bf Step 1.} The boundedness of the $(PS)_c$ sequence of $J_\psi$.
Let $\{u_n\}\subset H_0^2(\Omega)$ be a $(PS)_c$ sequence of $J_\psi$.
From the definition of the $(PS)_c$ sequence, we have, as $n\rightarrow\infty$,
\begin{equation}\label{equ4.6}
J_\psi(u_n)\rightarrow c, \ \ J_\psi'(u_n)\rightarrow 0 \ in \ H^{-2}(\Omega).
\end{equation}
Then it follows from $(f1)$, \eqref{th4.1}, \eqref{88} and $a<c_0<\dfrac{\theta}{2}a$ that
\begin{equation*}
\begin{split}
c+1+o(1)\|u_n\|=&J_\psi(u_n)-\frac{1}{\theta}\langle J_\psi'(u_n),u_n\rangle\\
=&(\frac{1}{2}-\frac{1}{\theta})\int_{\Omega}|\Delta u_n|^2\mathrm{d}x+\frac{1}{2}\widehat{\psi_0}(\int_{\Omega}|\nabla u_n|^2\mathrm{d}x)-\frac{1}{\theta}\psi_0(\int_{\Omega}|\nabla u_n|^2\mathrm{d}x)\int_{\Omega}|\nabla u_n|^2\mathrm{d}x\\
&+\lambda\int_{\Omega}(\frac{1}{\theta}f(x,u_n)u_n-F(x,u_n))\mathrm{d}x
+(\frac{1}{\theta}-\frac{1}{2^{**}})\int_{\Omega}|u_n|^{2^{**}}\mathrm{d}x\\
\geq&(\frac{1}{2}-\frac{1}{\theta})\int_{\Omega}|\Delta u_n|^2\mathrm{d}x+(\frac{1}{2}a-\frac{1}{\theta}c_0)\int_{\Omega}|\nabla u_n|^2\mathrm{d}x\\
\geq&(\frac{1}{2}-\frac{c_0}{\theta a})\|u_n\|^2,
\end{split}
\end{equation*}
which ensures the boundedness of $\{u_n\}$ in $H_0^2(\Omega)$.

{\bf Step 2.} The compactness of $\{u_n\}$ in $L^{2^{**}}(\Omega)$.
Since $\{u_n\}$ is bounded in $H_0^2(\Omega)$, as was done in Section $2$,
we still have \eqref{un}, \eqref{uv1} and \eqref{uv}. According to \eqref{equ4.6},
we have
\begin{equation}\label{ojun}
\begin{split}
o(1)=&\langle J_\psi'(u_n),u_n\phi\rangle\\
=&\int_{\Omega}\Delta{u_n}\Delta(u_n \phi)\mathrm{d}x+\psi_0(\int_{\Omega}|\nabla u_n|^2\mathrm{d}x)\int_{\Omega}\nabla {u_n}\nabla (u_n \phi)\mathrm{d}x-\lambda\int_{\Omega}f(x,u_n)u_n \phi\mathrm{d}x\\
&-\int_{\Omega}|u_n|^{2^{**}-2}u_n u_n \phi\mathrm{d}x\\
=&\int_{\Omega}|\Delta u_n|^2 \phi\mathrm{d}x+2\int_{\Omega}\Delta u_n\nabla u_n \nabla\phi\mathrm{d}x
+\int_{\Omega}\Delta u_n u_n \Delta\phi\mathrm{d}x+\psi_0(\int_{\Omega}|\nabla u_n|^2\mathrm{d}x)\int_{\Omega}|\nabla u_n|^2\phi\mathrm{d}x\\
&+\psi_0(\int_{\Omega}|\nabla u_n|^2\mathrm{d}x)\int_{\Omega}\nabla u_n u_n \nabla\phi\mathrm{d}x
-\lambda\int_{\Omega}f(x,u_n)u_n \phi\mathrm{d}x-\int_{\Omega}|u_n|^{2^{**}}\phi\mathrm{d}x,
\end{split}
\end{equation}
where $\phi$ is the same as that in the proof of Lemma \ref{jubuPStiaojian}.
Notice that each term on the right hand side of \eqref{ojun} is the same as that of \eqref{123}
except the fourth and fifth ones. The fourth term is nonnegative which can be dropped when we take limit,
and the fifth one can be estimated exactly the same as that of \eqref{ine3} by the boundedness of $\psi_0$.
Therefore, by letting $n\rightarrow\infty$ and $\varepsilon\rightarrow 0$ successively
in \eqref{ojun}, recalling \eqref{uv1}, \eqref{ine1}, \eqref{ine2}, \eqref{ine3}, \eqref{ine4}
and the nonnegativity of the fourth term on the right hand side of \eqref{ojun},
one has $\mu_i\leq \nu_i$. Combing this with \eqref{uv}, we arrive at
$$(\rm I) \ \ \nu_i=0 \ \ \ or \ \ \ (\rm {II}) \ \ \nu_i\geq S^\frac{N}{4}.$$

Next we claim that (\rm II) cannot occur. If (\rm II) holds for some $i \in I$,
according to $(f1)$, \eqref{uv1}, \eqref{uv}, \eqref{ine4.1}, \eqref{88}, \eqref{equ4.6}
and the inequality $a<c_0<\frac{\theta}{2}a$, we obtain
\begin{equation*}
\begin{split}
c=&\lim\limits_{n\rightarrow\infty}\left(J_\psi(u_n)-\frac{1}{\theta}\left\langle J_\psi'(u_n),u_n\right\rangle\right)\\
=&\lim\limits_{n\rightarrow\infty}\bigg\{(\frac{1}{2}-\frac{1}{\theta})\int_{\Omega}|\Delta u_n|^2\mathrm{d}x+\frac{1}{2}\widehat{\psi_0}(\int_{\Omega}|\nabla u_n|^2\mathrm{d}x)-\frac{1}{\theta}\psi_0(\int_{\Omega}|\nabla u_n|^2\mathrm{d}x)\int_{\Omega}|\nabla u_n|^2\mathrm{d}x\\
&+\lambda\int_{\Omega}(\frac{1}{\theta}f(x,u_n)u_n-F(x,u_n))\mathrm{d}x
+(\frac{1}{\theta}-\frac{1}{2^{**}})\int_{\Omega}|u_n|^{2^{**}}\mathrm{d}x\bigg\}\\
\geq&\lim\limits_{n\rightarrow\infty}\bigg\{(\frac{1}{2}-\frac{1}{\theta})\int_{\Omega}|\Delta u_n|^2\phi\mathrm{d}x+(\frac{1}{\theta}-\frac{1}{2^{**}})\int_{\Omega}|u_n|^{2^{**}}\phi\mathrm{d}x\bigg\}\\
\geq&(\frac{1}{2}-\frac{1}{\theta})\mu_i+(\frac{1}{\theta}-\frac{1}{2^{**}})\nu_i\\
\geq&(\frac{1}{2}-\frac{1}{\theta})S\nu_i^{\frac{2}{2^{**}}}+(\frac{1}{\theta}-\frac{1}{2^{**}})\nu_i\\
\geq&(\frac{1}{2}-\frac{1}{2^{**}})S^\frac{N}{4}.
\end{split}
\end{equation*}
This is impossible. Consequently, $\nu_i=0$ for all $i\in I$, which ensures that
$$\int_{\Omega}|u_n|^{2^{**}}\mathrm{d}x\rightarrow\int_{\Omega}|u|^{2^{**}}\mathrm{d}x\ \text{as} \ n\rightarrow\infty.$$

{\bf Step 3.} The compactness of $\{u_n\}$ in $H_0^2(\Omega)$. From the above discussion, we know that the $(PS)_c$ sequence $\{u_n\}$ is bounded in $H_0^2(\Omega)$ and there exist a subsequence of $\{u_n\}$
(which we still denote by $\{u_n\})$ and a function $u\in H_0^2(\Omega)$ such that
$$u_n\rightarrow u \ \ in \ L^{2^{**}}(\Omega).$$
According to \eqref{equ4.6}, we have
\begin{equation}\label{equality4.8}
\begin{split}
o(1)=&\langle J_\psi'(u_n),u_n-u\rangle\\
=&\int_{\Omega}|\Delta u_n|^2\mathrm{d}x-\int_{\Omega}\Delta u_n\Delta u\mathrm{d}x+\psi_0(\int_{\Omega}|\nabla u_n|^2\mathrm{d}x)\int_{\Omega}\nabla u_n\nabla(u_n-u)\mathrm{d}x\\
&-\lambda\int_{\Omega}f(x,u_n)(u_n-u)\mathrm{d}x-\int_{\Omega}|u_n|^{2^{**}-2}u_n(u_n-u)\mathrm{d}x.
\end{split}
\end{equation}
It follows from \eqref{ine4.1} and the strong convergence of $\{u_n\}$ in $H_0^1(\Omega)$ that
\begin{equation*}
\begin{split}
&\limsup\limits_{n\rightarrow\infty}\left|\psi_0(\int_{\Omega}|\nabla u_n|^2\mathrm{d}x)\int_{\Omega}\nabla u_n\nabla(u_n-u)\mathrm{d}x\right|\\
\leq&\limsup\limits_{n\rightarrow\infty}c_0\left|\int_{\Omega}|\nabla u_n|^2\mathrm{d}x-\int_{\Omega}\nabla u_n\nabla u\mathrm{d}x\right|\\
=&0.
\end{split}
\end{equation*}
Then combining this fact with \eqref{funu} and \eqref{un-u}, and letting $n\rightarrow\infty$
in \eqref{equality4.8}, we conclude that
$$\int_{\Omega}|\Delta u_n|^2\mathrm{d}x-\int_{\Omega}\Delta u_n\Delta u\mathrm{d}x=o(1), \ \ as \ n\rightarrow\infty,$$
which, together with \eqref{un}, implies $\|u_n-u\|\rightarrow0$ as $n\rightarrow\infty$.
The proof is complete.
\end{proof}

{\bf Proof of Theorem \ref{th4.1}.}
It follows from Lemma \ref{le4.1} that $J_\psi$ satisfies the mountain pass geometry,
which ensures that there exists a sequence $\{u_n\}\subset H_0^2(\Omega)$ such that
$J_\psi(u_n)\rightarrow c_\lambda$ and $J_\psi'(u_n)\rightarrow 0$ in $H^{-2}(\Omega)$
as $n\rightarrow \infty$, where
$$c_\lambda:=\inf_{\gamma\in \Gamma}\max_{t\in [0,1]}J_\psi(\gamma(t)),\ \Gamma=\{\gamma\in C([0,1],H_0^2(\Omega)):\gamma(0)=0,\gamma(1)=v_0\}.$$
and $v_0$ is determined in the proof of Lemma \ref{le4.1}. According to Lemmas \ref{le4.2}
and \ref{le4.3}, there exists $\lambda_0>0$ such that $J_\psi$ satisfies $(PS)_{c_\lambda}$
condition for all $\lambda>\lambda_0$. From Lemma \ref{Mountain Pass Lemma}, we know that there
exists a $u_\lambda\in H_0^2(\Omega)$ such that $u_n\rightarrow u_\lambda$ in $H_0^2(\Omega)$ as $n\rightarrow\infty$, and $u_\lambda$ is a nontrivial solution to problem \eqref{eq2} for all $\lambda>\lambda_0$. This proof of Theorem \ref{th4.1} is complete.

On the basis of Theorem \ref{th4.1}, we prove Theorem \ref{th1.2}.

{\bf Proof of Theorem \ref{th1.2}.}
Let $\lambda_0$ be as in Theorem \ref{th4.1}. Then for each $\lambda> \lambda_0$,
let $u_\lambda$ be a nontrivial solution to problem \eqref{eq2} and $c_\lambda$ be
the critical value of $J_\psi$. As was mentioned at the beginning of this section,
$u_\lambda$ is also a solution to problem \eqref{eq1} when $\|\nabla u_\lambda\|_2^2\leq t_0$.

We claim that there exists $ \lambda_*>\lambda_0$ such that $\|\nabla u_\lambda\|_2^2\leq t_0$ for all $\lambda\geq \lambda_*$.
If not, there would be a sequence $\{\lambda_n\}\subset \mathbb{R}$ with $\lambda_n\rightarrow+\infty$ as $n\rightarrow \infty$
and $\|\nabla u_{\lambda_n}\|_2^2> t_0$. From the proof of Theorem \ref{th4.1}, we see that $J_\psi(u_{\lambda_n})=c_{\lambda_n}$ and
$J_\psi'(u_{\lambda_n})=0$. Thus, by using $(f1)$, \eqref{ine4.1} and \eqref{88}, we deduce that
\begin{equation*}
\begin{split}
c_{\lambda_n}=&J_\psi(u_{\lambda_n})-\frac{1}{\theta}\langle J_\psi'(u_{\lambda_n}),u_{\lambda_n}\rangle\\
=&(\frac{1}{2}-\frac{1}{\theta})\int_{\Omega}|\Delta u_{\lambda_n}|^2\mathrm{d}x+\frac{1}{2}\widehat{\psi_0}(\int_{\Omega}|\nabla u_{\lambda_n}|^2\mathrm{d}x)-\frac{1}{\theta}\psi_0(\int_{\Omega}|\nabla u_{\lambda_n}|^2\mathrm{d}x)\int_{\Omega}|\nabla u_{\lambda_n}|^2\mathrm{d}x\\
&+\lambda_n\int_{\Omega}(\frac{1}{\theta}f(x,u_{\lambda_n})u_{\lambda_n}-F(x,u_{\lambda_n}))\mathrm{d}x
+(\frac{1}{\theta}-\frac{1}{2^{**}})\int_{\Omega}|u_{\lambda_n}|^{2^{**}}\mathrm{d}x\\
\geq&(\frac{a}{2}-\frac{c_0}{\theta})\int_{\Omega}|\nabla u_{\lambda_n}|^2\mathrm{d}x\\
>&(\frac{a}{2}-\frac{c_0}{\theta})t_0.
\end{split}
\end{equation*}
This contradicts with Lemma \ref{le4.2} for sufficiently large $n$. Consequently,
there exists $\lambda_*>\lambda_0>0$ such that $u_\lambda$ is a nontrivial solution to problem \eqref{eq1} for all $\lambda\geq \lambda_*$.

Finally, we study the asymptotic behavior of this solution $u_\lambda$ when $\lambda$ goes to infinity.
According to $(f1)$, \eqref{ine4.1} and \eqref{88}, we have
\begin{equation*}
\begin{split}
c_\lambda=J_\psi(u_\lambda)-\frac{1}{\theta}\langle J_\psi'(u_\lambda),u_\lambda\rangle
\geq(\frac{1}{2}-\frac{c_0}{\theta a})\|u_\lambda\|^2,
\end{split}
\end{equation*}
which, by recalling Lemma \ref{le4.2} again, implies that $\lim\limits_{\lambda\rightarrow\infty} \|u_\lambda\|=0$.
The proof of Theorem \ref{th1.2} is complete.


\begin{thebibliography}{xx}

\bibitem{Ball}
J.M. Ball, Initial-boundary value problems for an extensible beam, J. Math. Anal. Appl., {\bf42} (1973), 61-90.

\bibitem{Berger}
H.M. Berger, A new approach to the analysis of large deflections of plates. J. Appl. Mech., {\bf22} (1955), 465-472.

\bibitem{BreNi}
H. Br\'{e}zis, L. Nirenberg, Positive solutions of nonlinear elliptic equations involving critical Sobolev exponents,  Comm. Pure Appl. Math., {\bf36} (1983), 437-477.

\bibitem{DD1990}
D.E. Edmunds, D. Fortunato, E. Jannelli, Critical exponents, critical dimensions and the biharmonic operator, Arch. Rational Mech. Anal., {\bf112}(1990), 269-289.

\bibitem{Figueiredo2013}
G.M. Figueiredo, Existence of a positive solution for a Kirchhoff problem type with critical growth via truncation argument, J. Math. Anal. Appl., {\bf401}(2013), 706-713.

\bibitem{HE2016}
E.M. Hssini, M. Massar, N. Tsouli, Solutions to Kirchhoff equations with critical exponent,
Arab J. Math. Sci., {\bf22}(2016), 138-149.

\bibitem{LvZongyan2021}
Q. He, Z. Lv, Existence and nonexistence of nontrivial solutions for critical biharmonic equations,
J. Math. Anal. Appl., {\bf495}(2021), No. 124713, 30 pp.


\bibitem{LHW2023}
Q. Li, Y. Han, T. Wang, Existence and nonexistence of solutions to a critical biharmonic equation with logarithmic perturbation. Journal of Differential Equations, {\bf 365}(2023), 1-37.


\bibitem{Lions1}
P.L. Lions, The concentration-compactness principle in the calculus of variations. The
limit case, part 1, Rev. Mat. Iberoamericana, {\bf1}(1985), 145-201.

\bibitem{Lions2}
P.L. Lions, The concentration-compactness principle in the calculus of variations. The
limit case, part 2, Rev. Mat. Iberoamericana, {\bf1}(1985), 45-121.

\bibitem{ND}
D. Naimen, Positive solutions of Kirchhoff type elliptic equations involving a critical Sobolev exponent,
NoDEA Nonlinear Differential Equations Appl., {\bf 21}(2014), 885-914.

\bibitem{Shishaoyun}
Y. Song, S. Shi, Multiplicity of solutions for fourth-order elliptic equations of Kirchhoff type with critical exponent, J. Dyn. Control Syst., {\bf23}(2017), 375-386.

\bibitem{Wang and An}
F. Wang, Y. An, Existence and multiplicity of solutions for a fourth-order elliptic equation, Bound. Value
Probl., {\bf2012}(2012), No. 6, 9 pp.

\bibitem{WangFanglei}
F. Wang, M. Avci, Y. An, Existence of solutions for fourth order elliptic equations of Kirchhoff type, J. Math. Anal. Appl., {\bf409}(2014), 140-146.

\bibitem{Willem}
M. Willem, Minimax Theorems, Birkh\"{a}user, Boston, 1996.

\bibitem{RP1993}
R.C.A.M. Van der Vorst, Best constant for the embedding of the space $H^2\cap H_0^1$ into $L^{\frac{2N}{N-2}}$, Differential Integral Equations, {\bf6}(1993), 259-276.

\end{thebibliography}
\end{document}